\documentclass[preprint,12pt]{elsarticle}
\usepackage[top=3cm, bottom=3cm, left=2.2cm, right=2.2cm]{geometry}
\usepackage{amssymb}    % 数学符号
\usepackage{amsmath}    % 基础公式环境
\usepackage{amsthm}     % 定理、引理、定义环境
\usepackage{mathtools}  % 复杂公式扩展（分段函数、积分等）
\usepackage{cases}      % 优化分段函数排版
\usepackage{hyperref}   % 可点击交叉引用
\usepackage{multirow}  % 用于单元格跨行合并
\usepackage{array}     % 增强表格兼容性，避免语法错误
\usepackage{makecell}  % 单元格内换行核心包（新增）
\usepackage{ragged2e}  % 左右对齐+自动换行优化（必加载）
\usepackage{adjustbox}   % 核心：自动缩放表格到页面宽度（不拉伸文字）

\usepackage[table]{xcolor}

\newtheorem{theorem}{Theorem}[section]    % 定理：1.1, 2.1...
\newtheorem{lemma}[theorem]{Lemma}        % 引理：与定理共享编号
\newtheorem{definition}[theorem]{Definition}  % 定义：与定理共享编号
\newtheorem{proposition}[theorem]{Proposition} % 命题：与定理共享编号
\newtheorem{remark}[theorem]{Remark} 
\newtheorem{corollary}[theorem]{Corollary}
\newcommand{\R}{{\mathbb R}}

\begin{document}

\begin{frontmatter}

%% 标题（填写你的论文标题）
\title{Sharp threshold for a one-dimensional thin film equation in the supercritical case}

%% 作者（多个作者用 \and 分隔）
\author{Shen Bian \corref{cor1}}
\cortext[cor1]{Corresponding author. E-mail: bianshen66@163.com} % 通讯作者

%% 作者单位（多个单位用 label 区分，按需添加）
\affiliation{organization={Department of Mathematics, Beijing University of Chemical Technology},
            addressline={No. 15 East North Third Ring Road, Chaoyang District},
            city={Beijing},
            postcode={100029},
            country={China}}

%% 摘要（必填，控制在 150-200 词，概括研究目的、方法、结果）

\begin{abstract}
We study a one-dimensional thin film equation combining competitive effects of aggregation and repulsion, where repulsion is modeled by fourth-order diffusion and aggregation by backward second-order degenerate diffusion with exponent $m>0$. Under natural regularity constraints, we prove that for every $m>0$, there exists a unique (up to the mass-critical case $m=3$) nonnegative, radially decreasing steady state $U_*$ which coincides with the extremal function of the sharp Sz.-Nagy inequality and is simultaneously the global minimizer of the free energy. Using this variational characterization in the supercritical regime $3<m<\infty$, we show that finite-time blow-up occurs for all initial data whose initial free energy lies below the positive threshold $F(U_*)$, provided the $L^{m+1}$-norm of the initial datum exceeds that of $U_*$. Conversely, if the $L^{m+1}$-norm is below that of $U_*$, the solution exists globally and its second moment diverges as $t\to\infty$. This sharp criterion significantly extends the previously known blow-up condition requiring negative free energy to a much wider class of initial data (see \cite{BP00}). Our results identify the steady state $U_*$ as the critical pivot linking variational structure to dynamical behavior, and provide a constructive method to determine blow-up versus global existence via an explicit $L^{m+1}$-norm comparison.
\end{abstract}

% 35K65 \sep 34K21 \sep 35A01 \sep 35B44 \sep 35B38
\end{frontmatter}

%% 正文开始

\section{Introduction}
The fourth-order thin film equation is a central model in lubrication theory, describing the evolution of a thin viscous film on a solid substrate. It takes the form
\begin{align}\label{1}
\frac{\partial u}{\partial t}=-\frac{\partial}{\partial x} \left( u^n u_{xxx}\right)_x-\frac{\partial}{\partial x}\left(u^m u_x \right)_x,
\end{align}
where $u(x,t)$ denotes the local film thickness. The exponent $n>0$ is the mobility exponent, which encodes the boundary condition at the liquid-solid interface. Typical values are $n=3$ for the no-slip condition and $n=1$ for the perfect-slip (no-shear) condition characteristic of Hele-Shaw flows 
\cite{Tom1,Tom23,Tom25,Tom36}. The exponent $m$ arises from the disjoining pressure or a thickness-dependent body force. The case $0<m<3$ corresponds to non-classical Van der Waals forces or power-law molecular interactions \cite{Tom62,Tom65,Tom67}. The value $m=3$ recovers the classical Van der Waals force, which also appears together with $n=3$ in the description of fluid droplets hanging from a ceiling \cite{Tom26}. $m>3$ describes faster-decaying forces that localize the destabilizing effect to extremely thin films.                                                                                                                                                                                                                                                                                                                                                    

In this work, we focus on equation \eqref{1} with $n=1, m>0$, and consider nonnegative solutions of the following Cauchy problem
\begin{align}\label{2}
\left\{
  \begin{array}{ll}
 u_t=- (u u_{xxx})_x - (u (u^m)_x)_x, & x \in \R, ~t>0, \\
u(x,0)=u_0(x) \ge 0, & x \in \R.
\end{array}
\right.
\end{align}
Throughout this paper, initial data will be assumed to be
\begin{align}\label{initial}
u_0 \in L^1(\R) \cap H^1(\R),\qquad \int_{\R} x^2 u_0(x) dx< \infty. 
\end{align}
A fundamental property of the non-negative solutions to \eqref{2} is the formal conservation of the total mass:
\begin{align}\label{mass}
M:=\int_{\R} u_0(x) dx=\int_{\R} u(x,t) dx,\quad t \ge 0.
\end{align}
Equation \eqref{2} is governed by a fundamental competition between two nonlinear effects: the fourth-order term enforces stabilization and repulsion, whereas the backward second-order diffusion term drives destabilizing aggregation. This is well represented by the free energy
\begin{align}
F(u)=\frac{1}{2}\int_{\R} \left(u_x\right)^2 dx-\frac{1}{m+1} \int_{\R} u^{m+1} dx
\end{align}
satisfying the dissipation principle
\begin{align}
\frac{d}{dt}F(u(\cdot,t))=-\int_{\R} u \left((u_{xx}+u^m)_x\right)^2dx \le 0,
\end{align}
which can be derived by multiplying \eqref{2} by $\mu=-u_{xx}-u^m$. The main property of this energy is that it consists of two terms: the gradient energy $\frac{1}{2} \int_{\R} \left(u_x\right)^2 dx$ and the potential energy $-\frac{1}{m+1} \int_{\R} u^{m+1} dx$. 

Let us mention that \eqref{2} possesses a scaling invariance. If $u(x,t)$ is a solution to \eqref{2}, then the scaled function $u_\lambda(x,t)=\lambda^{\frac{2}{m-1}}u\left(\lambda x,\lambda^{\frac{4m-2}{m-1}} t\right)$ is also a solution to \eqref{2}. For the critical case $m=3$, the scaling reduces to the mass-invariant scaling $u_\lambda=\lambda u(\lambda x,\lambda^{5}t)$, under which the second-order aggregative term $\lambda^{m+3}  (u_{\lambda} (u_\lambda^m)_x)_x$ and the fourth-order repulsive term $\lambda^{6} (u_\lambda (u_\lambda)_{xxx})_x$ balance. For the subcritical case $0<m<3$, aggregation dominates repulsion for small $\lambda$, which prevents spreading, while for large $\lambda$, repulsion dominates aggregation, thus precluding blow-up. On the contrary, for the supercritical case $m>3$, aggregation dominates repulsion for large $\lambda$, and finite-time blow-up may occur. For small $\lambda$, repulsion dominates aggregation, allowing infinite-time spreading. This type of competition between aggregation and repulsion is a common feature in many nonlinear models, including Hele-Shaw flow, stellar collapse, chemotaxis models, and the nonlinear Schrödinger equation (see \cite{BL13,BCL09,BCKS99,B02,C03,C67,F01,KS70,M89,P07,SS99,W83}). In these models, the precise balance of the competing mechanisms critically governs the dynamics of solutions, ultimately determining the dichotomy between global existence and finite-time blow-up. 

For the subcritical case $0<m<3$, this is the repulsion-dominated regime. Global existence of solutions to \eqref{2} for any initial data was proved in \cite{BP98}. For the critical case $m=3$, this is the fair competition regime where repulsion and aggregation balance. \cite{WBB04} identified a critical mass $M_c$ and studied infinite-time self-similar spreading below the critical mass and finite-time blow-up self-similar solutions above the critical mass. Further analysis of blowing-up symmetric self-similar solutions with zero contact angle has been done in \cite{Jose60}. Both expanding and blowing-up self-similar solutions were also explored in \cite{Jose37,Jose38,Jose59}. Global existence of non-negative weak solutions was proved in \cite{LW17DC} with initial mass less than $M_c$.

For the supercritical case $m>3$, this is the aggregation-dominated regime. It is known from \cite{BP00} that finite-time blow-up occurs for initial data with negative free energy. However, to the best of our knowledge, little is known about the global existence of solutions. Our main goal in this paper is to establish a sharp threshold that distinguishes global existence from finite-time blow-up for solutions to \eqref{2}. To this end, we first present two theorems concerning stationary solutions, which are essential for the analysis. Based on these results, we will then show that finite-time blow-up actually occurs for all initial data whose initial free energy is below $F(U_*)$ with $F(U_*)>0$. This significantly extends the previously known blow-up criterion originally requiring negative free energy to a much wider class of initial data (see Theorem \ref{globalblow} below). The main tool for the analysis of \eqref{2} is a special case of the Sz.-Nagy inequality \cite{Sz.-Nagy41}.
\begin{lemma}[Sz.-Nagy inequality \cite{Sz.-Nagy41}]
Let the nonnegative function $f \in L^1(\R)$ and $f_x \in L^2(\R)$. Then $f \in L^{m+1}(\R)$ for all $m>0$, and we have the estimate
\begin{align}\label{Nagy}
\|f\|_{m+1}^{\frac{3(m+1)}{m}} \le C_* \|f\|_{1}^{\frac{m+3}{m}} \|f_x\|_2^2,
\end{align}
where the sharp constant is given by 
\begin{align}
C_*=\left( \frac{3}{4} Q\left( \frac{3}{2m},\frac{1}{2}  \right) \right)^2,\quad Q(x,y)=\frac{(x+y)^{-(x+y)\Gamma(x+y)}}{x^{-x}y^{-y}\Gamma(x)\Gamma(y)}.
\end{align}
In particular, for $m=3$, the sharp constant $C_*=\frac{9}{4\pi^2}$.
\end{lemma}

Before presenting our main results, we first fix the notation to be used throughout the paper. For a given mass, we introduce the set
\begin{align}\label{YM}
\mathcal{Y}_M:=\{ u \in L_+^1(\R) \cap H^1(\R): \|u\|_1=M \}.
\end{align}
In the mass-critical case $m=3$, we define
\begin{align}\label{YMc}
\mathcal{Y}_{M_c}:=\{ u \in L_+^1(\R) \cap H^1(\R): \|u\|_1=M_c \},
\end{align}
where
\begin{align}\label{Mc}
M_c:=\left( \frac{2}{C_*} \right)^{1/2}, \quad C_*=\frac{9}{4 \pi^2}.
\end{align}
For $m \neq 3$, we define
\begin{align}\label{YMPstar}
\mathcal{Y}_{M,P_\ast}:=\{ u \in \mathcal{Y}_M: \|u\|_{m+1}=P_\ast \},
\end{align}
where
\begin{align}\label{Pstar}
P_\ast:=\left( \frac{3(m+1)}{2 m C_\ast M^{\frac{m+3}{m}}} \right)^{\frac{m}{(m+1)(m-3)}}.
\end{align}

To investigate the dynamical behavior of solutions to \eqref{2}, we require several properties of steady states. Steady states of the equation under consideration have already been studied in the literature. Positive periodic steady states in $C^3(\R)$ and nonnegative steady states with connected support and either zero or nonzero contact angles were investigated in \cite{Tom42,Tom43}. In the present work, the steady states relevant to our dynamical analysis are neither periodic nor of the equal-contact-angle type; instead, they belong to a different regularity class, namely $L^1 \cap H^1(\R)$. The critical dynamical threshold is closely related to the extremal functions of the Sz.-Nagy inequality. \textbf{Our first main result} establishes that these extremal functions are indeed steady states. Specifically, they are nonnegative, radially symmetric, and non-increasing. The precise statement is given below.

\begin{theorem}\label{Nagyoptimi}
Let $0<m<\infty$. Then the following statements hold:
\begin{itemize}
  \item[(i)] For $m \neq 3$, the Sz.-Nagy inequality \eqref{Nagy} admits a unique optimizer $U_*$ in $\mathcal{Y}_{M,P_\ast}$. This optimizer is the unique nonnegative, radially decreasing steady state of \eqref{2} with bounded support and zero contact angle.
  \item[(ii)] For $m=3$, the Sz.-Nagy inequality \eqref{Nagy} admits infinitely many optimizers in $\mathcal{Y}_{M_c}$. These optimizers are nonnegative, radially decreasing steady states of \eqref{2} and form a one-parameter family under mass-invariant scaling.
\end{itemize}
\end{theorem}

We next turn to the global minimization of the free energy. Existing results in \cite{Tom44} show that positive periodic steady states act as energy-unstable saddle points, and that families of compactly supported steady states possess lower energy than their periodic counterparts. Building on this insight, the present work goes further by identifying the exact energy minimizers. In particular, we prove that the optimizers of the Sz.-Nagy inequality coincide with the global minimizers of the free energy. \textbf{Our second main result} characterizes the associated energy landscape as follows:

\begin{theorem}\label{fumini}
Let $0<m<\infty$. Then the following statements hold:
\begin{itemize}
\item[(i)] For $3<m<\infty$, $U_\ast$ is the unique global minimizer of $F(u)$ over $\mathcal{Y}_{M,P_\ast}$, and is also the unique global minimizer of $F(u)$ among all steady-state solutions to \eqref{2} with mass $M$. Moreover, we have
         \begin{align}
            \inf_{u \in \mathcal{Y}_{M,P_\ast}} F(u)=F(U_\ast)=\frac{m-3}{3(m+1)} P_\ast^{m+1}>0.
         \end{align}
      \item[(ii)] For $m=3$, any optimizer $u_c$ of the Sz.-Nagy inequality \eqref{Nagy} is a global minimizer of $F(u)$ over
      $\mathcal{Y}_{M_c}$ and satisfies $F(u_c)=0$.
      \item[(iii)] For $0<m<3$, $U_\ast$ is the unique global minimizer of $F(u)$ over $\mathcal{Y}_M$. Moreover, we have
             \begin{align}
                 \inf_{u \in \mathcal{Y}_M} F(u)=F(U_\ast)=\frac{m-3}{3(m+1)} P_\ast^{m+1}<0.
             \end{align}
\end{itemize}
\end{theorem}

Before proceeding, let us state the notion of weak solutions adopted throughout this paper.
\begin{definition}[Weak solution]\label{weakdefine}
Let $u_0$ be initial data satisfying \eqref{initial} and $T \in (0,\infty]$. A weak solution $u$ to \eqref{2} is a nonnegative function satisfying
\begin{align*}
& u \in L^\infty \bigl(0,T; L^1(\R)\cap H^1(\R) \bigr), \\
& u \in L^2 \bigl(0,T; H^2(\R) \bigr), \\
& u^{1/2} u_{xxx} \in L^2\bigl(0,T;L^2(\R)\bigr)
\end{align*}
such that for every $0<t<T$ and any $\psi \in C_0^\infty(\R)$,
 \begin{align}
 \int_{\R} \psi u(\cdot,t)\,dx- \int_{\R} \psi u_0(x)\,dx & =\int_0^t
 \int_{\R} \psi_x u u_{xxx} \,dx\,ds-\frac{m}{m+1} \int_0^t \int_{\R} \psi_{xx}\,u^{m+1}\,dx\,ds.  \label{weak}
 \end{align}
This weak solution is also a weak entropy solution satisfying
 \begin{align}
 F(u(\cdot,t))+\int_0^t \int_{\R} u \left| (u_{xx}+u^m)_x \right|^2 dxds \le F(u_0),\quad 0<t<T.
 \end{align}
\end{definition}

We are now ready to present \textbf{our third main result} concerning the dynamical behavior of solutions in the supercritical regime.
\begin{theorem}\label{globalblow}
For $3<m<\infty$, assume $F(u_0)<F(U_\ast)=\frac{m-3}{3(m+1)}P_*^{m+1}$.
 \begin{enumerate}
    \item[(i)] If $\|u_0\|_{m+1}>\|U_\ast\|_{m+1}$, then the solution blows up in finite time.
    \item[(ii)] If $\|u_0\|_{m+1}<\|U_\ast\|_{m+1}$, then the solution exists globally in time, and its second moment satisfies
    \begin{align*}
    \lim_{t \to \infty} \int_{\R} x^2 u(\cdot,t) dx=+\infty.
    \end{align*}
 \end{enumerate}
\end{theorem}

\begin{remark}
\begin{itemize}
\item[1.] Our results provide a sharp dichotomy criterion for global existence versus finite-time blow-up of solutions in the supercritical regime $3<m<\infty$. The main difficulty we address arises from the fact that, unlike the critical case, mass conservation alone does not provide sufficient scaling control. Our approach overcomes this issue by fully exploiting the variational structure of steady-state solutions and its intrinsic connection to the free energy functional.
\item[2.] We note that under $F(u_0)<F(U_\ast)$, the equality $\|u_0\|_{m+1}=\|U_\ast\|_{m+1}$ cannot hold; see Remark \ref{Fu0FUstar}. For $F(u_0)>F(U_\ast)$, two scenarios remain possible: either the second moment of the solution diverges, or its $L^{m+1}$-norm remains unbounded. If both quantities are bounded, the solution may converge to $U_\ast$. We discuss these issues further in Section \ref{FU0great}.
\end{itemize}
\end{remark}

\begin{remark}[Physical interpretation: potential well and threshold dynamics]
The above result reveals sharp threshold behavior governed by the potential-well structure of gradient-flow dynamics:
\begin{itemize}
    \item \textbf{Initial energy below the steady-state energy} $F(u_0) < F(U_\ast)$ implies the system has already ``crossed the ridge''. Its evolution proceeds only \textbf{downhill (energy decreasing)} and cannot overcome the energy barrier.
    \item \textbf{The initial $L^{m+1}$ norm} $\|u_0\|_{m+1}$ determines \textbf{which downhill path} the system follows.
    \begin{itemize}
        \item[(a)] If $\|u_0\|_{m+1} < \|U_\ast\|_{m+1}$, the system enters a \textbf{gentle valley inside the potential well}, where $\|u(\cdot,t)\|_{m+1}$ and $\|u_x\|_{2}$ remain uniformly bounded. Consequently, the free energy is bounded from below and decreases toward a limit $F_\infty<F(U_\ast)$. The second moment is necessarily unbounded, implying mass escapes to spatial infinity and the solution decays locally on every compact set. Whether the solution profile, after suitable translation, converges to a steady state with energy lower than $F(U_\ast)$ remains an open question; it may instead converge to a singular measure at infinity.
        \item[(b)] If $\|u_0\|_{m+1} > \|U_\ast\|_{m+1}$, the system slides into a \textbf{steep ravine outside the potential well}, where attractive forces dominate, gradients grow rapidly, and finite-time blow-up (rupture) occurs.
    \end{itemize}
\end{itemize}
\end{remark}

The contributions of this work are threefold: to characterize the variational structure of the free energy, to identify ground-state solutions as extremals of the Sz.-Nagy inequality, and to establish a sharp threshold separating global existence from finite-time blow-up.

Throughout this paper, we denote by $C$ a generic positive constant that may vary from line to line, and write $C=C(\cdot,\cdots,\cdot)$ to indicate a constant depending only on the parameters inside the parentheses.

\section{Existence criterion}\label{sec2}
In this section, we establish the local existence and blow-up criteria for bounded weak solutions to the nonlinear equation \eqref{2}. To this end, we first define the standard mollifier $J(x)\in C^\infty(\R)$ by 
\begin{align*}
J(x)=\left\{
       \begin{array}{ll}
      C e^{\frac{1}{|x|^2-1}}, & |x|<1, \\
      0, & |x| \ge 1,
       \end{array}
     \right.
\end{align*}
where the constant C is normalized so that $\int_{\R} J(x) dx=1$.This mollifier will be used in the subsequent compactness analysis and limiting procedure. Inspired by \cite{LW17DC}, we consider the regularized problem.
\begin{align}\label{regularized}
\left\{
  \begin{array}{ll}
    \partial_t u_\varepsilon+\partial_x \left( \sqrt{u_\varepsilon^2+\varepsilon^2} \left( (u_\varepsilon)_{xxx}+(u_\varepsilon^m)_x  \right) \right)=0, & x\in \R, t>0, \\
  u_\varepsilon(x,0)=u_{0\varepsilon}(x) \ge 0, & x \in \R,
  \end{array}
\right.
\end{align}
where $u_{0\varepsilon}(x)=J_\varepsilon * u_0(x)$ with $J_\varepsilon(x)=\frac{1}{\varepsilon} J\left( \frac{x}{\varepsilon} \right)$ such that 
\begin{align}
\left\{
  \begin{array}{ll}
    u_{0\varepsilon} \ge 0, \quad \|u_{0\varepsilon}\|_1=\|u_0\|_1, \\
    u_{0\varepsilon}(x) \to u_0(x)\text{ in } L^q(\R) \text{ for } 1 \le q<\infty, \text{ as } \varepsilon \to 0, \\
    \int_{\R} x^2 u_{0\varepsilon}(x) dx \to \int_{\R} x^2 u_{0}(x) dx, \text{ as } \varepsilon \to 0.
  \end{array}
\right.
\end{align}
The results in \cite[Proposition 2]{LW17DC} and \cite[Theorem 3]{LW17DC} assert that for any $0<t<T$, if 
\begin{align}\label{starstar}
\|u_\varepsilon(\cdot,t)\|_{H^1(\R)} \le C_0,
\end{align}
where $C_0$ is independent of $\varepsilon$, then there exists a subsequence $\varepsilon_n \to 0$ such that
\begin{align}
u_{\varepsilon_n} \rightharpoonup{*} u \text{ in } L^\infty \left( 0,T;H^1(\R) \right), \\
u_{\varepsilon_n} \rightharpoonup{} u \text{ in } L^2 \left( 0,T;H^2(\R) \right). 
\end{align}
Furthermore, $u$ is a non-negative weak entropy solution satisfying
\begin{align}
F(u(\cdot,t))+\int_0^t \int_{\R} u \left| u_{xxx}+(u^m)_x \right|^2 dxds \le F(u_0),\quad 0<t<T.
\end{align}
According to the above analysis, a weak entropy solution to \eqref{2} on $[0,T)$ exists when \eqref{starstar} is fulfilled. Therefore, we focus on establishing the $H^1-$bound which additionally provides a characterisation of the maximal existence time. 

\begin{proposition}[Local in time existence and blow-up criteria]\label{local}
Under assumption \eqref{initial} on the initial condition, there exist a maximal existence time $T_w \in (0,\infty]$ and a free energy solution $u$ to \eqref{2} on $[0,T_w).$ If $T_w<\infty,$ then
\begin{align}
\|u_x(\cdot,t)\|_{L^2(\R)} \to \infty \text{ as } t \to T_w.
\end{align}
\end{proposition}
\begin{proof}
The main task in proving the local existence is to obtain the following $H^1$-estimates. By multiplying equation \eqref{regularized} by $-(u_\varepsilon)_{xx}$, we obtain
\begin{align*}
\frac{1}{2}\frac{d}{dt}\int_{\R} (u_\varepsilon)_x^2 dx & =-\int_{\R} (u_\varepsilon)_{xxx}^2 \sqrt{u_\varepsilon^2+\varepsilon^2}dx-\int_{\R} (u_\varepsilon)_{xxx} (u_\varepsilon^m)_x \sqrt{u_\varepsilon^2+\varepsilon^2}dx \\
& \le -\frac{1}{2} \int_{\R} (u_\varepsilon)_{xxx}^2 \sqrt{u_\varepsilon^2+\varepsilon^2}dx+\frac{1}{2} \int_{\R} (u_\varepsilon^m)_x^2 \sqrt{u_\varepsilon^2+\varepsilon^2} dx \\
& \le C \|u_\varepsilon\|_{L^\infty(\R)}^{2m-1} \int_{\R} (u_\varepsilon)_x^2 dx \\
& \le C \left( \int_{\R} (u_\varepsilon)_x^2 dx+K_1(M) \right)^{\frac{2m+1}{2}},
\end{align*}  
where $K_1(M)$ is a constant depending on $M$, and we have used the fact that the $H^1$ norm controls $L^\infty$ in one space dimension, that is
\begin{align}
\|u_\varepsilon\|_{L^\infty(\R)}^2 \le \|u_\varepsilon\|_2^2+\|(u_\varepsilon)_x\|_2^2,
\end{align} 
combining with 
\begin{align}
\|u_\varepsilon\|_2^2 \le \|u_\varepsilon\|_1 \|u_\varepsilon\|_{L^\infty(\R)} \le \frac{1}{2}\|u_\varepsilon\|_1^2+\frac{1}{2}\|u_\varepsilon\|_{L^\infty(\R)}^2.
\end{align}
Hence, we have
\begin{align}
\frac{d}{dt}\int_{\R} (u_\varepsilon)_x^2 dx \le \overline{C} \left( \int_{\R} (u_\varepsilon)_x^2 dx+K_1(M) \right)^{\frac{2m+1}{2}}.
\end{align}
Solving the above inequality we obtain
\begin{align}
\int_{\R} (u_\varepsilon)_x^2 dx+K_1(M) \le \frac{1}{\left( \left(\int_{\R} (u_\varepsilon)_x^2 dx+K_1(M)\right)^{\frac{1-2m}{2}}-\frac{2m-1}{2}\overline{C} t \right)^{\frac{2}{2m-1}}}.
\end{align}
This implies that there exists a $T_0=T_0\left(\|u_0\|_{H^1(\R)} \right)$ independent of $\varepsilon$ such that
\begin{align}\label{H1}
\|u_\varepsilon\|_{L^\infty(0,T_0;H^1(\R))} \le C.
\end{align}
Then, by a word-for-word translation of \cite[Lemma 2]{LW17DC}, we obtain 
\begin{align}\label{H2}
\|u_\varepsilon\|_{L^2(0,T_0;H^2(\R))} \le \tilde{C},
\end{align}
where $\tilde{C}$ is a constant depending only on $\|u_0\|_{L^1(\R)}$ and $\|u_0\|_{L^\infty(\R)}$. 

As a consequence, the a priori bounds \eqref{H1} and \eqref{H2} hold true uniformly in $\varepsilon$ and thus we can pass to the limit by the Lions-Aubin lemma, which provides time compactness. In the limit, we obtain the local existence of a weak solution $u$ to \eqref{2}. Furthermore, in light of the proof for the maximal existence time in \cite[Theorem 2.4]{BCL09}, we deduce that there is $T_w \in (0,\infty]$ such that either $T_w=\infty$ or $T_w<\infty$ and $\|u_x(\cdot,t)\|_{L^2(\R)} \to \infty$ as $t \to T_w$. This completes the proof. 
\end{proof}

\section{Proofs of Theorems \ref{Nagyoptimi} and \ref{fumini}}\label{sec3}

This section constructs the variational framework needed for analyzing global existence and finite-time blow-up in the following section. We first study steady-state solutions of \eqref{2} and their relation to the free energy functional. Subsection \ref{subsec1} presents basic properties of steady states as necessary preliminaries. 

Based on these preparations, we further analyze the Euler-Lagrange equation derived from the Sz.-Nagy inequality. We prove that the extremal functions of this inequality are steady-state solutions of \eqref{2}, which completes the proof of Theorem \ref{Nagyoptimi}. We then show that these extremal functions are global minimizers of the free energy under suitable constraints, which yields the proof of Theorem \ref{fumini}. Detailed proofs are given in Subsections \ref{subsec2} and \ref{subsec3}, respectively.

\subsection{Steady States}\label{subsec1}
This subsection is devoted to a systematic investigation of steady-state solutions to equation \eqref{2}. In \cite{Tom43}, Laugesen and Pugh considered positive steady states in $C^3(\R)$, as well as compactly supported touchdown steady states with regularity $C^1[c,d] \cap C^3(c,d)$. They proved that periodic positive steady states and touchdown steady states with equal contact angles satisfy
\begin{align}
U U_{xxx}+U (U^m)_x=0 \quad \text{in } \R.
\end{align}
In this work, we focus on solutions belonging to the function class \(L^1(\R)\cap H^1(\R)\). Such solutions are distinct from the periodic profiles and equal-contact-angle steady states investigated in previous literature. Instead, they correspond to critical points of the free energy under mass conservation, characterized by a constant chemical potential over their support and compact support consisting of finitely many intervals, with no additional constraints imposed on their contact angles.

Let us start by precisely defining steady states to \eqref{2}.
\begin{definition}
Given $U_s \in L_+^1(\R) \cap H^1(\R)$, we call it a steady state of \eqref{2} if $U_s (U_s^m)_x \in H^1_{loc}(\R)$, $(U_s)_{xxx} \in L_{loc}^1(\R)$, and it satisfies
\begin{align}\label{fivestar}
U_s (U_s)_{xxx}+U_s (U_s^m)_x=0 \quad \text{in } \R
\end{align}
in the sense of distributions. Throughout this subsection, we assume the following:
\begin{itemize}
\item[$(K1)$] $0<m<\infty$, $U_s \in L^{1}(\R) \cap H^1(\R) \cap C(\R)$ with $U_s \ge 0$ and $\int_{\R} U_s dx=M$.
\item[$(K2)$] Let $\Omega:=\{ x \mid U_s>0 \}$ and $U_s=0$ in $\R \setminus \Omega$.
\item[$(K3)$] If $\Omega$ is unbounded, then $U_s$ decays sufficiently fast at infinity.
\item[$(K4)$] If $\Omega$ is bounded, then $\partial \Omega \in C^1$ and $U_s \in C^1(\overline{\Omega})$.
\end{itemize}
\end{definition}

We now present equivalent characterizations of steady states and show that the chemical potential is constant on each connected component of the support.

\begin{proposition}\label{prop1}
Under assumptions $(K1)$--$(K4)$, let
\begin{align}
\mu_s=-(U_s)_{xx}-U_s^m.
\end{align}
Then the following statements are equivalent:
\begin{itemize}
\item[$(i)$] No dissipation: $\displaystyle\int_{\R} U_s \left| (\mu_s)_x  \right|^2 dx=0$.
\item[$(ii)$] $\Omega=\displaystyle \bigcup_{k=1}^N (a_k,b_k)$ with $1 \le N<\infty$ is a finite union of bounded intervals, and
   \begin{align}
   \mu_s=C_k \quad \text{for } x \in (a_k,b_k),
   \end{align}
where $C_k$ is a constant on each interval $(a_k,b_k)$.
\end{itemize}
\end{proposition}

\begin{proof}
$(i)\Rightarrow(ii)$. Since $U_s \in H^1(\R)$, the one-dimensional embedding implies $U_s \in L^\infty(\R)$. Since $C_0^\infty(\R)$ is dense in $H^1(\R)$, we obtain
\begin{align}
0=\int_{\R} \mu_s \left( U_s (\mu_s)_x \right)_x dx=-\int_{\R} U_s |(\mu_s)_x|^2 dx.
\end{align}

$(ii)\Rightarrow(i)$. Since $U_s \in C(\R)$, the set $\Omega$ is open and consists of countably many disjoint open components, which may be unbounded. We write $\Omega=\bigcup_{k=1}^N (a_k,b_k)$. Where $U_s>0$, equation \eqref{fivestar} reduces to
\begin{align}
(U_s)_{xxx}+(U_s^m)_x=0
\end{align}
in the sense of distributions on $\Omega$. It follows that
\begin{align}
(U_s)_{xx}+U_s^m=C_k \quad \text{on } (a_k,b_k).
\end{align}
Different intervals may have different constants $C_k$.

We first claim that $\Omega$ has no unbounded component. Suppose for contradiction that there exists an index $k_0$ such that $(a_{k_0},b_{k_0})=(a,\infty)$. On this interval,
\begin{align*}
(U_s)_{xx}+U_s^m=C_{k_0}, \quad x \in (a,\infty).
\end{align*}
Since $U_s \to 0$ as $|x|\to\infty$, we have $U_s^m \to 0$ and $(U_s)_{xx} \to C_{k_0}$ as $|x|\to\infty$. This contradicts $U_s \in L^1(\R)$ unless $U_s \equiv 0$.

Next, we claim that $1 \le N<\infty$. Suppose, for contradiction, that $N=\infty$. Then all intervals are contained in a bounded set, so $L_k:=b_k-a_k \to 0$. Rescaling each interval to $(0,1)$ via $x=a_k+L_k y$ and setting $v_k(y)=U_s(a_k+L_k y)$, we obtain
\begin{align}
v_k''+L_k^2 v_k^m=L_k^2 \mu_k, \quad v_k(0)=v_k(1)=0, \ v_k>0.
\end{align}
As $L_k \to 0$, the equation formally becomes $v_k''=0$, whose only solution with zero boundary conditions is $v_k\equiv0$, contradicting $v_k>0$. Thus $N$ must be finite, and $\Omega$ is a finite union of bounded disjoint open intervals.
This completes the proof.
\end{proof}

A particularly important class is obtained by requiring that $U_s$ be a critical point of the free energy $F(u)$ with fixed mass $\int_{\R} u \,dx=M$. As we show below, this additional condition forces the chemical potential $\mu_s$ to be the same constant on all connected components.

\begin{proposition}\label{prop2}
Assume $(K1)$--$(K4)$. Let $U_s$ be a critical point of $F(u)$ subject to $\int_{\R}U_s \,dx=M$. Then
\begin{itemize}
\item[$(i)$] The following Pohozaev identity holds:
 \begin{align}\label{poho}
   \int_{\R} \left|(U_s)_x \right|^2 dx=\frac{2m}{3(m+1)} \int_{\R} U_s^{m+1} dx.
 \end{align}
\item[$(ii)$] The chemical potential satisfies
\begin{align}
 (U_s)_{xx}+U_s^m=\overline{C} \quad \text{a.e. } x \in \Omega,
\end{align}
where $\overline{C}=\dfrac{1}{M}\dfrac{m+3}{3(m+1)} \int_{\R} U_s^{m+1} dx$.
\item[$(iii)$] The following boundary condition holds:
\begin{align}\label{bdd}
 \sum_{k=1}^N \left( b_k (U_s)_x^2(b_k)-a_k (U_s)_x^2(a_k) \right)=0.
\end{align}
In particular, if $\Omega=[-L,L]$, then $U_s'(L)=U_s'(-L)=0$.
\end{itemize}
\end{proposition}

\begin{proof}
By definition, $U_s$ is a critical point of $F(u)$ under the mass constraint $\int_{\R} U_s \,dx=M$. That is, for every $\eta \in C_0^\infty(\R)$ satisfying $\int_{\R} \eta dx=0$,
\begin{align}\label{Fvar}
    \frac{d}{d \varepsilon} \bigg |_{\varepsilon=0} F(U_s+\varepsilon \eta)=0.
\end{align}

\emph{(i) Pohozaev identity.} Consider the scaling ansatz $U_\lambda(x)=\lambda U_s(\lambda x)$ for $\lambda>0$. Since $U_s$ is a critical point, the derivative of $F(U_\lambda)$ at $\lambda=1$ vanishes; this can be verified by taking $\eta=U_s+x (U_s)_x$, which has zero integral. A direct computation yields
\begin{align}
\left. \frac{d F(U_\lambda)}{d \lambda} \right|_{\lambda=1}
=\frac{3}{2} \int_{\R} | (U_s)_x |^2 dx-\frac{m}{m+1} \int_{\R} U_s^{m+1} dx = 0.
\end{align} 
This gives the Pohozaev identity
\begin{align}\label{0501}
\frac{3}{2} \int_{\R} |(U_s)_x|^2 dx=\frac{m}{m+1} \int_{\R} U_s^{m+1} dx. 
\end{align}

\emph{(ii) Global constant chemical potential.} With the Pohozaev identity in hand, we further show that the chemical potential is globally constant. Choose two disjoint connected components $\Omega_i=(a_i,b_i)$ and $\Omega_j=(a_j,b_j)$, and take nonnegative functions $\phi_i \in C_0^\infty(\Omega_i)$, $\phi_j \in C_0^\infty(\Omega_j)$ such that $\int_{\Omega_i} \phi_i dx=\int_{\Omega_j} \phi_j dx$. Define $\eta=\phi_i-\phi_j$; clearly $\int_{\R} \eta \,dx=0$. Substituting $\eta$ into \eqref{Fvar}, we obtain
\begin{align*}
0 &=\int_{\R} \left((U_s)_x \eta_x-U_s^m \eta \right) dx \\
&= \int_{\Omega_i} \left((U_s)_x (\phi_i)_x-U_s^m \phi_i \right) dx- \int_{\Omega_j} \left((U_s)_x (\phi_j)_x-U_s^m \phi_j \right) dx \\
&=-\int_{\Omega_i} \big((U_s)_{xx}+U_s^m\big) \phi_i dx+\int_{\Omega_j} \big((U_s)_{xx}+U_s^m\big) \phi_j dx \\
&=-C_i \int_{\Omega_i} \phi_i dx+C_j \int_{\Omega_j} \phi_j dx.
\end{align*}
Hence $C_i=C_j$. Denote this common constant by $\overline{C}$; then
\begin{align}\label{0502}
(U_s)_{xx}+U_s^m=\overline{C}\quad \text{in } \Omega.
\end{align}
Multiplying \eqref{0502} by $U_s$ and integrating over $\R$ gives
\begin{align}\label{0503}
\int_{\R} U_s^{m+1} dx-\int_{\R} |(U_s)_x|^2 dx=\overline{C} M.
\end{align}
Combining \eqref{0501} and \eqref{0503}, we arrive at
\begin{align*}
\overline{C}=\frac{1}{M} \frac{m+3}{3(m+1)} \int_{\R} U_s^{m+1} dx.
\end{align*}

\emph{(iii) Boundary conditions.} Recall that $\Omega=\bigcup_{k=1}^N (a_k,b_k)$. Multiplying \eqref{0502} by $x (U_s)_x$ and by $U_s$, respectively, we get
\begin{align*}
\frac{1}{2} x (U_s)_x^2 \bigg |_{a_k}^{b_k}- \frac{1}{2}\int_{a_k}^{b_k} (U_s)_x^2 dx-\frac{1}{m+1}\int_{a_k}^{b_k} U_s^{m+1} dx
=-\overline{C} \int_{a_k}^{b_k}U_s dx
\end{align*}
and 
\begin{align*}
\int_{a_k}^{b_k}(U_s)_x^2 dx-\int_{a_k}^{b_k} U_s^{m+1} dx
=-\overline{C} \int_{a_k}^{b_k}U_s dx.
\end{align*}
Combining these yields
\begin{align*}
\frac{3}{2} \sum_{k=1}^N \int_{a_k}^{b_k} (U_s)_x^2 dx
=\frac{m}{m+1} \sum_{k=1}^N \int_{a_k}^{b_k} U_s^{m+1} dx
+\frac{1}{2} \sum_{k=1}^N \left( b_k (U_s)_x^2(b_k)-a_k (U_s)_x^2(a_k) \right).
\end{align*}
Applying \eqref{0501}, we conclude
\begin{align*}
\sum_{k=1}^N \left( b_k (U_s)_x^2(b_k)-a_k (U_s)_x^2(a_k) \right)=0.
\end{align*}
In particular, when $N=1$ with $\Omega=[-L,L]$, we have
\begin{align*}
L (U_s)_x^2(-L)+L (U_s)_x^2(L)=0,
\end{align*}
which forces $U_s'(-L)=U_s'(L)=0$. The proof is complete.
\end{proof}

Using the Pohozaev identity, we now derive the free energy of steady states.

\begin{corollary}
The free energy of a steady state with mass $M$ satisfies
\begin{align}\label{Fus}
F(U_s)=\frac{m-3}{3(m+1)} \|U_s\|_{m+1}^{m+1}
\left\{
\begin{array}{ll}
>0, & m>3, \\
=0, & m=3, \\
<0, & 0<m<3.
\end{array}
\right.
\end{align}
\end{corollary}

\subsection{Proof of Theorem \ref{Nagyoptimi}} \label{subsec2}

This subsection is devoted to the proof of Theorem \ref{Nagyoptimi}. We establish a direct correspondence between steady-state solutions of \eqref{2} and the extremal functions of the Sz.-Nagy inequality. This relation will be used to identify the critical profile that determines the threshold between global existence and finite-time blow-up. We first recall the Euler-Lagrange equation satisfied by extremals of the Sz.-Nagy inequality, which was established in \cite{Sz.-Nagy41}.

\begin{lemma}[Extremals of the Sz.-Nagy inequality \cite{Sz.-Nagy41}]\label{prop3}
Let $0<m<\infty$ and $\alpha=\frac{m+3}{m}$. For any $u \in L^1(\R) \cap H^1(\R)$, equality in \eqref{Nagy} is attained if and only if
$u=\lambda V(\mu |x-x_0|)$ for some $\lambda>0$, $\mu>0$, and $x_0 \in \R$,
where $V \ge 0$ is radially symmetric, non-increasing, and solves the Euler-Lagrange equation
\begin{align}\label{VV2}
-V_{xx}-\frac{\alpha+2}{2} \frac{\|V\|_{m+1}^{\alpha+1-m}}{\|V\|_1^\alpha C_*} V^m
=
-\frac{\alpha}{2}\frac{\|V\|_{m+1}^{\alpha+2}}{\|V\|_1^{\alpha+1} C_*}
\quad \text{a.e. in } \operatorname{supp}(V).
\end{align}
\end{lemma}

By choosing a suitable normalization of the $L^{m+1}$-norm, the extremal functions can be transformed into steady states of \eqref{2}. We now present the proof of Theorem \ref{Nagyoptimi}.

\begin{proof}[Proof of Theorem \ref{Nagyoptimi}.] The proof can be divided into two steps. 
\textbf{Step 1} (Proof of (i)). For $m \neq 3$, we recall the definition of $P_\ast$ from \eqref{Pstar}:
\begin{align}
P_\ast=\left( \frac{\alpha+2}{2 C_* M^\alpha} \right)^{\frac{1}{m-\alpha-1}}
=\left( \frac{3(m+1)}{2m C_* M^\alpha} \right)^{\frac{m}{(m-3)(m+1)}}.
\end{align}
For a given solution $V_0(x)$ to \eqref{VV2}, we define
\begin{align}\label{WW1}
U_\ast(x)=\frac{1}{\lambda_\ast} V_0 \left( \frac{x}{\mu_\ast} \right),
\end{align}
where
\begin{align}
\lambda_\ast=\left( \frac{M}{\|V_0\|_1} \frac{\|V_0\|_{m+1}^{m+1}}{P_\ast^{m+1}} \right)^{1/m},\quad
\mu_\ast= \frac{\lambda_\ast M}{\|V_0\|_1}.
\end{align}
Then $U_\ast$ is also a solution to \eqref{VV2} satisfying
\begin{align}
-(U_\ast)_{xx}=U_\ast^m-\frac{1}{M}\frac{m+3}{3(m+1)} \|U_\ast\|_{m+1}^{m+1}
\quad \text{a.e. in } \operatorname{supp}(U_\ast)
\end{align}
with
\begin{align}\label{MPstar}
\|U_\ast\|_1=M,\quad \|U_\ast\|_{m+1}=P_\ast.
\end{align}
This shows that $U_\ast$ is a radially symmetric, non-increasing steady-state solution to \eqref{2}.

Next, applying Theorem \ref{prop2}(iii), we conclude that for $0<m<\infty$, $U_\ast$ solves the free boundary problem
\begin{align}\label{radialequation}
\left\{
  \begin{array}{ll}
   -u''=u^m-\frac{1}{M}\frac{m+3}{3(m+1)} P_\ast^{m+1}, \quad 0<r<L, \\[2mm]
    u'(0)=0,\quad u(L)=u'(L)=0,
  \end{array}
\right.
\end{align}
with $0<L<\infty$. It follows from \cite[Theorem 3]{PS98} that this problem admits a unique solution. This completes the proof of (i).

{\it\textbf{Step 2}} (Proof of (ii)). For $m=3$, we have $\alpha+2=m+1$. Then \eqref{VV2} reduces to
\begin{align}\label{VV3}
-V_{xx}= \frac{2}{\|V\|_1^2 C_*} V^m-\frac{\|V\|_{m+1}^{m+1}}{\|V\|_1^{3} C_*} \quad \text{a.e. in } \operatorname{supp}(V).
\end{align}
Hence, a solution $V_c$ to \eqref{VV3} is also a steady-state solution to \eqref{2} if and only if
\begin{align}
\|V_c\|_1 = M_c=\left( \frac{2}{C_*} \right)^{1/2}, \quad C_*=\frac{9}{4 \pi^2}.
\end{align}
Thus, $V_c$ satisfies
\begin{align}\label{VV4}
-(V_c)_{xx}=V_c^m-\frac{1}{2}\frac{\|V_c\|_{m+1}^{m+1}}{M_c} \quad \text{a.e. in } \operatorname{supp}(V_c).
\end{align}
Observe that equation \eqref{VV4} is scaling invariant: if $V_c$ is a solution, then $\mu V_c(\mu x)$ is also a solution for any $\mu>0$ with the same mass $M_c$. Therefore, the solutions to \eqref{VV4} form a one-parameter family of radial steady states. This completes the proof of Theorem \ref{Nagyoptimi}.
\end{proof}

\subsection{Proof of Theorem \ref{fumini}}\label{subsec3}

This subsection is dedicated to the proof of Theorem \ref{fumini}. We demonstrate that the extremal functions of the Sz.-Nagy inequality act as ground states, which are global minimizers of the free energy under suitable constraints. The proof is split into three cases based on the range of the exponent m: Proposition \ref{picture} addresses part (i), Proposition \ref{picture1} covers part (ii), and Proposition \ref{picture2} deals with part (iii). For subsequent analysis, we first define the set of steady-state solutions with prescribed mass $M$ as 
\begin{align}
\mathcal{S}_M := { u \in \mathcal{Y}_M : u \text{ is a steady-state solution to } \eqref{2} }.
\end{align}

\begin{proposition}[Aggregation-dominated regime]\label{picture}
Let $3<m<\infty$, $\alpha=\frac{m+3}{m}$, and let $P_\ast$ be defined by \eqref{Pstar}. Let $U_\ast$ be the unique optimizer given by Theorem \ref{Nagyoptimi}(i). Then the following statements hold:
\begin{itemize}
   \item[(i)] $U_\ast$ is the unique global minimizer of $F(u)$ in $\mathcal{Y}_{M,P_\ast}$.
   \item[(ii)] $U_\ast$ is the unique global minimizer of $F(u)$ in $\mathcal{S}_M$.
\end{itemize} 
Moreover, 
\begin{align}
\inf_{u \in \mathcal{Y}_{M,P_\ast}} F(u) = \inf_{u \in \mathcal{S}_{M}} F(u)= F(U_\ast)=\frac{m-3}{3(m+1)} P_\ast^{m+1}.
\end{align}   
\end{proposition} 

\begin{proof}
We first state two identities that follow directly from the Sz.-Nagy inequality \eqref{Nagy} and the Pohozaev identity \eqref{poho}:
\begin{align}
  u \text{ is an optimizer of the Sz.-Nagy inequality \eqref{Nagy} } &\Rightarrow \|u\|_{m+1}^{\alpha+2} = C_* \|u\|_1^\alpha \|u_x\|_2^2, \label{mpstar} \\
  u \text{ is a steady-state solution to \eqref{2} } &\Rightarrow \frac{m}{m+1} \|u\|_{m+1}^{m+1}=\frac{3}{2} \|u_x\|_2^2.  \label{sm}
\end{align}

\textbf{Step 1} (Proof of (i)). For a given mass $M>0$ and any $v \in \mathcal{Y}_{M,P_\ast}$, the Sz.-Nagy inequality \eqref{Nagy} yields
\begin{align}
F(v) & = \frac{1}{2}\|v_x\|_2^2-\frac{1}{m+1}\|v\|_{m+1}^{m+1} \nonumber \\
& \ge \frac{1}{2C_\ast M^\alpha} \|v\|_{m+1}^{\alpha+2}-\frac{1}{m+1}\|v\|_{m+1}^{m+1} \nonumber \\
& = \frac{1}{2C_\ast M^\alpha} P_\ast^{\alpha+2}-\frac{1}{m+1}P_\ast^{m+1} \nonumber \\
& = \frac{m-3}{3(m+1)} \|U_\ast\|_{m+1}^{m+1} = F(U_\ast),  \label{Fmin}
\end{align}
where we have used Corollary \ref{Fus}. Thus, $F(v) \ge F(U_\ast)$ for all $v \in \mathcal{Y}_{M,P_\ast}$. By \eqref{mpstar}, equality holds in \eqref{Fmin} if and only if $v=U_\ast$. This completes the proof of (i).

\textbf{Step 2} (Proof of (ii)). For any $w \in \mathcal{S}_M$, the Sz.-Nagy inequality \eqref{Nagy} and identity \eqref{sm} imply
\begin{align}
\|w\|_{m+1}^{\alpha+2} & \le C_* M^\alpha \|w_x\|_2^2 \nonumber \\
& = C_* M^\alpha \frac{2m}{3(m+1)} \|w\|_{m+1}^{m+1} \nonumber \\
&= \|w\|_{m+1}^{\alpha+2} \frac{\|w\|_{m+1}^{m-\alpha-1}}{P_\ast^{m-\alpha-1}}  = \|w\|_{m+1}^{\alpha+2} \frac{\|w\|_{m+1}^{m-\alpha-1}}{\|U_\ast\|_{m+1}^{m-\alpha-1}}. \label{m1min}
\end{align}
Since $m-\alpha-1>0$ for $m>3$, we conclude that $\|w\|_{m+1} \ge \|U_\ast\|_{m+1}$ for all $w \in \mathcal{S}_M$. It then follows from Corollary \ref{Fus} that
\begin{align}
F(w) =\frac{m-3}{3(m+1)} \|w\|_{m+1}^{m+1}  \ge \frac{m-3}{3(m+1)} \|U_\ast\|_{m+1}^{m+1}=F(U_\ast).
\end{align}
By \eqref{mpstar}, equality holds in \eqref{m1min} if and only if $w=U_\ast$. This completes the proof of (ii).
\end{proof}

\begin{proposition}[Fair competition regime]\label{picture1}
Let $m=3$ and let $M_c$ be defined by \eqref{Mc}. Then there exist infinitely many global minimizers of $F(u)$ over $\mathcal{Y}_{M_c}$. These minimizers are radially decreasing, compactly supported steady-state solutions $u_c$ to \eqref{2}. Moreover,
\begin{align}
\inf_{u \in \mathcal{Y}_{M_c}} F(u)=F(u_c)=0.
\end{align}
\end{proposition}

\begin{proof}
For any $u_c \in \mathcal{Y}_{M_c}$, the equality case of the Sz.-Nagy inequality \eqref{Nagy} yields
\begin{align}
F(u_c) & =\frac{1}{2}\|(u_c)_x\|_2^2-\frac{1}{4}\|u_c\|_{4}^{4}  \nonumber \\
& \ge \frac{1}{2} \frac{\|u_c\|_{4}^{4}}{C_\ast \|u_c\|_1^2}-\frac{1}{4}\|u_c\|_{4}^{4} \nonumber\\
& = \frac{1}{2} \frac{\|u_c\|_{4}^{4}}{C_\ast M_c^2}-\frac{1}{4}\|u_c\|_{4}^{4} = 0. \label{Fmass}
\end{align}
Equality holds in \eqref{Fmass} if and only if $u_c$ is a radially decreasing, compactly supported steady-state solution to \eqref{2} with mass $M_c$. This completes the proof.
\end{proof}

\begin{proposition}[Diffusion-dominated regime]\label{picture2}
Let $0<m<3$, $\alpha=\frac{m+3}{m}$, and let $U_\ast$ be the unique optimizer from Theorem \ref{Nagyoptimi}(i). Then $U_\ast$ is the unique global minimizer of $F(u)$ in $\mathcal{Y}_M$. Moreover,
\begin{align}
\inf_{u \in \mathcal{Y}_M} F(u)=F(U_\ast)=\frac{m-3}{3(m+1)} P_\ast^{m+1}.
\end{align}
\end{proposition}

\begin{proof}
For any $u \in \mathcal{Y}_M$, the Sz.-Nagy inequality \eqref{Nagy} implies
\begin{align}
F(u) & =\frac{1}{2}\|u_x\|_2^2-\frac{1}{m+1}\|u\|_{m+1}^{m+1}  \nonumber \\
& \ge \frac{1}{2} \frac{\|u\|_{m+1}^{\alpha+2}}{C_\ast M^\alpha}-\frac{1}{m+1}\|u\|_{m+1}^{m+1}. \label{sub}
\end{align}
Now define the auxiliary function
\begin{align}
g(x)=\frac{1}{2 C_\ast M^\alpha} x^{\alpha+2}-\frac{1}{m+1}x^{m+1}.
\end{align}
Since $0<m<3$, we have $\alpha+2>m+1$, so $g(x)$ attains its minimum at $x_*=P_\ast$. It then follows from \eqref{sub} that
\begin{align}\label{Fge}
F(u)=g\left( \|u\|_{m+1} \right) \ge \frac{1}{2 C_\ast M^\alpha} P_\ast^{\alpha+2}-\frac{1}{m+1} P_\ast^{m+1}=\frac{m-3}{3(m+1)} \|U_\ast\|_{m+1}^{m+1}=F(U_\ast).
\end{align}
By Theorem \ref{Nagyoptimi}(i), equality holds in \eqref{Fge} if and only if $u=U_\ast$. This completes the proof.
\end{proof}

Combining Propositions \ref{picture}, \ref{picture1}, and \ref{picture2}, we finish the proof of Theorem \ref{fumini}.

\section{Proof of Theorem \ref{globalblow}}\label{sec4}

This section is devoted to the proof of Theorem \ref{globalblow}.Using the characterization of the critical profile $U_\ast$ established in Theorem \ref{Nagyoptimi}, we analyze the dynamical behavior of solutions to \eqref{2} in the supercritical regime $3<m<\infty$.We establish a sharp dynamical threshold determined by the $L^{m+1}$-norm of $U_\ast$.The proof is divided into two parts:we prove finite-time blow-up in Subsection \ref{sub42} and global existence in Subsection \ref{sub41}. Specifically, initial data exceeding this threshold leads to finite-time blow-up, whereas initial data falling below the threshold ensures global existence of solutions.

\subsection{A priori estimates for $F(u_0)<F(U_\ast)$}

In this subsection, we derive a priori estimates under the energy condition $F(u_0)<F(U_\ast)$. These estimates will play a key role in the subsequent analysis of finite-time blow-up and global existence.

\begin{proposition}\label{belowabove}
Let $3<m <\infty$, $\alpha=\frac{m+3}{m}$, and let $U_\ast$ be the unique optimizer given by Theorem \ref{Nagyoptimi}(i). Assume that $u(x,t)$ is a weak entropy solution to \eqref{2} with initial data $u_0$ satisfying \eqref{initial}, \eqref{mass}, and
\begin{align}\label{Fu0}
 F(u_0)<F(U_\ast).
\end{align}
\begin{enumerate}
  \item[(i)] If
  \begin{align}\label{xiaoyum1}
  \|u_0\|_{m+1}< \|U_\ast\|_{m+1}, 
  \end{align}
  then there exists a constant $\mu_1<1$ such that the corresponding weak entropy solution $u$ satisfies
     \begin{align}\label{xiaoyum}
       \|u(\cdot,t)\|_{m+1} < \mu_1 \|U_\ast\|_{m+1}
    \end{align}
for all $t>0$.
  \item[(ii)] If
  \begin{align}\label{dayum1}
  \|u_0\|_{m+1} > \|U_\ast\|_{m+1},
  \end{align}
  then there exists a constant $\mu_2>1$ such that the corresponding weak entropy solution $u$ satisfies
\begin{align}\label{dayum}
     \|u(\cdot,t)\|_{m+1} > \mu_2 \|U_\ast\|_{m+1}
\end{align}
 for all $t>0$.
\end{enumerate}
\end{proposition}

\begin{proof}
First, applying the Sz.-Nagy inequality \eqref{Nagy}, we deduce from the definition of $F(u)$ that
\begin{align}\label{06081}
F(u) & = \frac{1}{2} \| u_x \|_2^2 -\frac{1}{m+1} \|u\|_{m+1}^{m+1} \nonumber \\
  & \ge \frac{1}{2C_* M^\alpha} \|u\|_{m+1}^{\alpha+2}-\frac{1}{m+1} \|u\|_{m+1}^{m+1}.
\end{align}
Moreover, substituting the optimizer $U_\ast$ into $F(u)$ yields
\begin{align}\label{06082}
F(U_\ast)& =\frac{1}{2} \|(U_\ast)_x\|_2^2-\frac{1}{m+1} \|U_\ast\|_{m+1}^{m+1} \nonumber \\
& = \frac{1}{2C_* M^\alpha} \|U_\ast\|_{m+1}^{\alpha+2}-\frac{1}{m+1} \|U_\ast\|_{m+1}^{m+1}.
\end{align}
We define the auxiliary function
\begin{align}\label{gx}
g(x)=\frac{1}{2C_* M^\alpha} x^{\alpha+2}-\frac{1}{m+1} x^{m+1}.
\end{align}
Combining \eqref{06081} and \eqref{06082}, we obtain
\begin{align}
g \left( \|U_\ast\|_{m+1} \right) =F(U_\ast) \ge g \left( \|u\|_{m+1} \right)
\end{align}
for all $u \in L^1(\R) \cap L^{m+1}(\R)$. Now, under the condition
\begin{align}\label{414}
F\left(u_0 \right)< F\left(U_\ast \right),
\end{align}
there exists a constant $0<\delta<1$ such that
\begin{align}
F\left(u_0 \right)<\delta F\left(U_\ast \right).
\end{align}
Since $F(u)$ is non-increasing in time, it follows that
\begin{align}\label{0608starstar}
g\left( \|u \|_{m+1} \right)\le F(u) \le F(u_0) < \delta F\left(U_\ast \right)=\delta g \left( \|U_\ast\|_{m+1} \right).
\end{align}

On the other hand, since $\alpha+2<m+1$ for $m>3$, a direct computation shows that the maximum of $g(x)$ is attained at
\begin{align*}
x_*=\left( \frac{\alpha+2}{2C_* M^\alpha} \right)^{\frac{1}{m-\alpha-1}}=\|U_\ast\|_{m+1}.
\end{align*}
This implies that $g(x)$ is strictly increasing for $x<\|U_\ast\|_{m+1}$. Therefore, if
\begin{align}
\|u_0\|_{m+1}< \|U_\ast\|_{m+1},
\end{align}
then \eqref{0608starstar} ensures that there exists $\mu_1<1$ (depending on $\delta$) such that, for all $t>0$,
\begin{align}
\|u \|_{m+1}< \mu_1 \|U_\ast\|_{m+1}.
\end{align}
Conversely, $g(x)$ is strictly decreasing for $x>\|U_\ast \|_{m+1}$. Thus, if
\begin{align}
\|u_0\|_{m+1}> \|U_\ast\|_{m+1},
\end{align}
using \eqref{0608starstar} again, we conclude that there exists $\mu_2>1$ (depending on $\delta$) such that
\begin{align}\label{420}
\|u \|_{m+1}> \mu_2 \|U_\ast\|_{m+1}.
\end{align}
This completes the proof.
\end{proof}

\begin{remark}\label{Fu0FUstar}
We note that if $F(u_0)<F(U_\ast)$, the case $\|u_0\|_{m+1}=\|U_\ast\|_{m+1}$ cannot occur. Indeed, it follows directly from Proposition \ref{picture}(i) that $F(u_0) \ge F(U_\ast)$ whenever $\|u_0\|_{m+1}=\|U_\ast\|_{m+1}=P_\ast$.
\end{remark}

\subsection{Finite-time blow-up} \label{sub42}

This subsection is devoted to the finite-time blow-up part of Theorem \ref{globalblow}. We prove that solutions blow up in finite time when the initial $L^{m+1}$-norm exceeds the critical threshold. The analysis relies on the time evolution of the second moment, following the classical approach proposed in \cite{JL92}. We define the second moment as
\begin{align}
m_2(t):=\int_{\R} x^2 u(x,t),dx.
\end{align}

\begin{theorem}[Finite-time blow-up]\label{finiteblowup}
Let $3<m<\infty$, and let $u(x,t)$ denote the weak entropy solution to \eqref{2} obtained from Proposition \ref{local} on $[0,T_w)$. Under assumptions \eqref{initial}, \eqref{Fu0}, and \eqref{dayum1}, the solution to \eqref{2} satisfies
\begin{align}\label{m22t}
\frac{d}{dt} m_2(t)=6 F(u)-2 \frac{m-3}{m+1} \int_{\R} u^{m+1} dx,\quad 0<t<T_w.
\end{align}
Moreover, $T_w<\infty$, and $u(x,t)$ blows up in finite time in the sense that
\begin{align*}
\limsup_{t \to T_w}\|u_x(\cdot,t)\|_{2}=\infty.
\end{align*}
\end{theorem}

\begin{proof}
We present the formal computation; the passage to the limit from the regularized problem \eqref{regularized} can be justified as in the proof of \cite[Theorem~3]{LW17DC} without additional difficulty. The key ingredient is the time evolution of the second moment. Integrating by parts in \eqref{2}, we compute
\begin{align}
\frac{d}{dt} m_2(t)
&= 2 \int_{\R} x u u_{xxx} dx+\frac{2m}{m+1} \int_{\R} x (u^{m+1})_x dx \nonumber \\
&=-2 \int_{\R} u u_{xx} dx-2\int_{\R} x u_x u_{xx} dx-\frac{2m}{m+1} \int_{\R} u^{m+1} dx \nonumber \\
&=2 \int_{\R} |u_x|^2 dx+ \int_{\R} |u_x|^2 dx  -\frac{2m}{m+1} \int_{\R} u^{m+1} dx \nonumber \\
&=3 \int_{\R} |u_x|^2 dx-\frac{2m}{m+1} \int_{\R} u^{m+1} dx \nonumber \\
&=6 F(u)-2 \frac{m-3}{m+1} \int_{\R} u^{m+1} dx. \label{m222t}
\end{align}
Since $m>3$, we use \eqref{m222t} together with \eqref{Fu0} and estimate \eqref{dayum} from Proposition \ref{belowabove} (with some $\mu_2>1$) to obtain
\begin{align}\label{m2t}
\frac{d}{dt} m_2(t)
& < 6 F(u_0)-2 \mu_2 \frac{m-3}{m+1} \int_{\R} U_\ast^{m+1} dx \nonumber \\
&<6 F(U_\ast)-2 \mu_2 \frac{m-3}{m+1} \int_{\R} U_\ast^{m+1} dx \nonumber \\
&=2(1-\mu_2)\frac{m-3}{m+1} \|U_\ast\|_{m+1}^{m+1}<0.
\end{align}
Here we have used that $F(u)$ is non-increasing in time and Corollary \ref{Fus}. It follows from \eqref{m2t} that there exists $T>0$ such that
\begin{align*}
\lim_{t \to T} m_2(t)=0.
\end{align*}

Next, by H\"{o}lder's inequality,
\begin{align}
\int_{\R} u(x) dx
&\le \int_{|x| \le R} u(x) dx+\int_{|x|>R} u(x) dx \nonumber \\
&\le C R^{\frac{m}{m+1}} \|u\|_{m+1} +\frac{1}{R^2} m_2(t).
\end{align}
Choosing $
R=\left( \frac{C m_2(t)}{\|u\|_{m+1}} \right)^{\frac{m+1}{3m+2}}
$ yields
\begin{align}
M \le C \|u\|_{m+1}^{\frac{2(m+1)}{3m+2}} m_2(t)^{\frac{m}{3m+2}},
\end{align}
which implies
\begin{align*}
\limsup_{t \to T} \|u(\cdot,t)\|_{m+1}
\ge \lim_{t \to T} M^{\frac{3m+2}{2(m+1)}} m_2(t)^{-\frac{m}{2(m+1)}}=\infty.
\end{align*}
Using the Sz.-Nagy inequality \eqref{Nagy}, we conclude that
\begin{align*}
\limsup_{t \to T} \|u_x \|_2=\infty.
\end{align*}
This completes the proof.
\end{proof}

\subsection{Global existence} \label{sub41}

This subsection is dedicated to the global existence part of Theorem \ref{globalblow}. We prove that solutions exist globally in time when the initial $L^{m+1}$-norm lies below the critical threshold associated with $U_\ast$.

\begin{theorem}[Global existence]\label{globalexistence}
Let $3<m<\infty$. Assume that \eqref{initial}, \eqref{Fu0}, and \eqref{xiaoyum1} hold. Then the weak solution obtained in Proposition \ref{local} exists globally on $[0,\infty)$ and satisfies
\begin{align}
\|u(\cdot,t)\|_{m+1}+\|u_x (\cdot,t)\|_2 \le C\bigl(\|u_0\|_1,\|(u_0)_x\|_{2} \bigr),\quad 0<t<\infty
\end{align}
and
\begin{align}
\int_0^t \int_{\R} u \left|u_{xxx} + (u^m)_x\right|^2 dxds \le C\bigl(\|u_0\|_1,\|(u_0)_x\|_{2} \bigr),\quad 0<t<\infty.
\end{align}
Furthermore,
\begin{align}\label{m2blow}
\lim_{t \to \infty} m_2(t)=\infty.
\end{align}
\end{theorem}

\begin{proof}
The proof consists of two parts. Step 1 establishes the global existence of weak solutions. Step 2 proves that the second moment tends to infinity, thereby verifying \eqref{m2blow}.

\textbf{Step 1} (Global existence). Under assumption \eqref{xiaoyum1}, it follows from \eqref{xiaoyum} in Proposition \ref{belowabove} that for any $t>0$, there exists $\mu_1<1$ such that
\begin{align}\label{260109}
\|u(\cdot,t)\|_{m+1}<\mu_1 \|U_\ast\|_{m+1}.
\end{align}
Furthermore, since the free energy $F(u)$ is non-increasing in time, we have for any $t>0$:
\begin{align}\label{260110}
\frac{1}{2} \|u_x \|_2^2  \le \frac{1}{m+1} \|u(\cdot,t)\|_{m+1}^{m+1}+ F(u_0) \le C\bigl(\|U_\ast\|_{m+1},F(u_0)\bigr),
\end{align}
and
\begin{align}
\int_0^t \int_{\R} u \left|u_{xxx} + (u^m)_x \right|^2 dxds \le F(u_0)+\frac{1}{m+1} \|u(\cdot,t)\|_{m+1}^{m+1} \le C\bigl(\|U_\ast\|_{m+1},F(u_0)\bigr).
\end{align}
Combining the Sobolev embedding theorem with H\"{o}lder's inequality for $1<r<\infty$, we infer from \eqref{260110} that
\begin{align}
\| u\|_{L^r(\R)} \le C, \quad \text{for any } 1 \le r \le \infty.
\end{align}
Therefore, for any $T>0$, using the arguments from Lemma~2 in \cite{LW17DC} in $H^2(\R)$, the following uniform regularity estimates hold:
\begin{align}
&u \in L^\infty(0,T;L^1 \cap H^1 (\R)), \label{11} \\
&u \in L^\infty(0,T;L^{r}(\R)) \quad \text{for any } 1 \le r \le \infty, \label{22} \\
& u \in L^2(0,T;H^2(\R)), \label{33} \\
&u^{1/2} u_{xxx} \in L^2(0,T;L^2(\R)), \nonumber \\
&u^{1/2} (u^m)_x \in L^2(0,T;L^2(\R)). \nonumber
\end{align}
In addition, the second moment satisfies
\begin{align}\label{m2bdd}
m_2(t) \le m_2(0)+6 F(u_0)t, \quad \text{for any } 0<t<T.
\end{align}
Thus, uniform bounds for the terms in the free energy have been established. Following the global existence argument in \cite{LW17DC}, we conclude for $3<m<\infty$ that there exists a weak entropy solution to \eqref{2} on $[0,T]$ for any $0<T<\infty$.

\textbf{Step 2} (Proof of \eqref{m2blow}). Suppose, for contradiction, that there exists an increasing sequence of times $\{t_k\}_{k \in \mathbb{N}} \to \infty$ such that
\begin{align}
m_2(t_k)=\int_{\R} x^2 u(x,t_k) dx
\end{align}
is bounded. Since $u \in L^1(\R) \cap H^1(\R)$, we claim that there exist a subsequence (still denoted $t_k$) and a function $u_\infty$ such that as $k \to \infty$,
\begin{align}
& u(\cdot,t_k) \rightharpoonup u_\infty \text{ in } H^1(\R), \nonumber \\
& u(\cdot,t_k) \to u_\infty \text{ in } L^1(\R), \label{L1} \\
& u(\cdot,t_k) \to u_\infty \text{ in } L^{m+1}(\R) \label{Lm1}
\end{align}
with $\|u_\infty\|_1=M$. Moreover, the second moment of the limit satisfies
\begin{align}
0<\int_{\R} x^2 u_\infty dx \le \liminf_{k \to \infty}\int_{\R} x^2 u(\cdot, t_k) dx<\infty,
\end{align}
since concentration at a Dirac delta function at the origin is excluded by the uniform bound \eqref{11}.

On the other hand, a direct consequence of \eqref{260109} is that the free energy $F(u)$ is bounded from below:
\begin{align}
F(u_0)-\liminf_{k \to \infty} F(u)(t_k) =\lim_{k \to \infty} \int_0^{t_k} \left( \int_{\R} u(x,s) \left| u_{xxx}+(u^m)_x \right|^2 dx \right) ds.
\end{align}
It follows that the Fisher information is integrable, and
\begin{align}
\lim_{k \to \infty} \int_{t_k}^\infty \left( \int_{\R} u(x,s) \left| u_{xxx}+(u^m)_x \right|^2 dx \right) ds=0,
\end{align}
which implies that, up to a subsequence, the limit $u_\infty$ as $k \to \infty$ satisfies
\begin{align}\label{260123}
\bigl((u_\infty)_{xx}+u_\infty^m\bigr)_x=0 \quad \text{a.e. in } \operatorname{supp}(u_\infty).
\end{align}
This follows by arguments similar to those in the existence proof in Section~2 of \cite{BCM08}. Therefore, $u_\infty$ is a radially symmetric steady-state solution to \eqref{2} with mass $M$. By Proposition \ref{picture}(ii),
\begin{align}\label{Usinfty}
F(U_\ast) \le F(u_\infty).
\end{align}
However, \eqref{Lm1} yields
\begin{align}
F(u_\infty) & =\frac{1}{2} \|(u_\infty)_x \|_{2}^2-\frac{1}{m+1} \|u_\infty\|_{m+1}^{m+1} \nonumber \\
& \le \liminf_{k \to \infty} \Bigl( \frac{1}{2} \|u_x(\cdot, t_k)\|_{2}^2-\frac{1}{m+1} \|u(\cdot,t_k)\|_{m+1}^{m+1} \Bigr) \nonumber \\
& = \liminf_{k \to \infty} F(u(\cdot,t_k)) \le F(u_0)<F(U_\ast),
\end{align}
contradicting \eqref{Usinfty}. Hence, \eqref{m2blow} holds.

It remains to justify the convergences \eqref{L1} and \eqref{Lm1} used above. We first prove \eqref{L1}. Let $0 \le \chi_R \le 1$ be a function in $C_0^\infty(\R)$ such that $\chi_R=1$ on $B_R(0)$ and $\chi_R=0$ on $\R\setminus B_{2R}(0)$. We decompose
\begin{align}
M=\int_{\R} u(\cdot,t_k) dx=\int_{\R} u(\cdot,t_k) \chi_R dx+\int_{\R} u(\cdot,t_k) (1-\chi_R) dx.
\end{align}
Then for sufficiently large $R$,
\begin{align}\label{largeR}
\int_{\R} u(\cdot,t_k) (1-\chi_R) dx \le \int_{|x| \ge R} u(\cdot, t_k) dx \le \frac{1}{R^2} \int_{|x| \ge R} |x|^2 u(\cdot,t_k) dx<\varepsilon,
\end{align}
while
\begin{align*}
\lim_{k \to \infty} \int_{\R} u(\cdot,t_k) \chi_R dx=\int_{\R} u_\infty \chi_R dx.
\end{align*}
Hence,
\begin{align}
\Bigl| M-\lim_{k \to \infty} \int_{\R} u(\cdot,t_k) \chi_R dx \Bigr|<\varepsilon \quad \text{as } R \to \infty,
\end{align}
which implies
\begin{align}\label{massconservation}
\int_{\R} u_\infty dx=M>0.
\end{align}
This proves \eqref{L1}.

We next prove \eqref{Lm1}. Using the condition $1<m+1<\infty$ and the compact Sobolev embedding $H^1(B_R(0)) \hookrightarrow L^q(B_R(0))$ for $1 \le q<\infty$, we obtain local strong convergence in $L^{m+1}(B_R(0))$ for any fixed $R>0$. Moreover, it follows from the Sz.-Nagy inequality \eqref{Nagy} that
\begin{align}
\int_{|x| \ge R} u^{m+1} dx
\le C_* \left( \int_{|x| \ge R} u dx  \right)^{\frac{3+m}{3}} \| u_x \|_{2}^{\frac{2}{3}m}
\le C\bigl(M,C_*,F(u_0)\bigr) \left( \frac{m_2(t)}{R^2} \right)^{\frac{3+m}{3}}
\end{align}
upon using \eqref{largeR}. Therefore, as $R \to \infty$ and $k \to \infty$,
\begin{align*}
&\int_{\R} \left|u(\cdot,t_k)-u_\infty \right|^{m+1} dx
=\int_{|x| < R} \left|u(\cdot,t_k)-u_\infty \right|^{m+1} dx+\int_{|x| \ge R} \left|u(\cdot,t_k)-u_\infty \right|^{m+1} dx \\
\le &\int_{|x| < R} \left|u(\cdot,t_k)-u_\infty \right|^{m+1} dx+C \int_{|x| \ge R} \bigl(|u(\cdot,t_k)|^{m+1}+|u_\infty|^{m+1}\bigr) dx \to 0.
\end{align*}
This completes the proof of Theorem \ref{globalexistence}.
\end{proof}

\subsection{Possible dynamic behaviors for $F(u_0)>F(U_\ast)$}\label{FU0great}

For initial data satisfying $F(u_0) > F(U_\ast)$, energy dissipation no longer prevents solutions from approaching the steady state $U_\ast$. The long-time behavior of global solutions is conjectured to depend on the boundedness of two key quantities:
\begin{align*}
&\textbf{(P1) } \sup_{t \ge 0} \| u(\cdot,t)\|_{m+1} < \infty, \quad \text{(concentration control)} \\
&\textbf{(P2) } \sup_{t \ge 0} \int_{\R} x^2 u(x,t)\,dx < \infty. \quad \text{(mass confinement)}
\end{align*}
Based on the gradient-flow structure and analogy with related models, we expect the following scenarios:
\begin{itemize}
    \item \textbf{(P1) and (P2):} The solution is expected to converge (possibly up to translation) to $U_\ast$.
    \item \textbf{(P1) and not (P2):} Mass is expected to escape to infinity while the profile remains bounded in $L^{m+1}$. After a suitable time-dependent translation that follows the escaping core, the profile may converge locally to a steady state.
    \item \textbf{not (P1) and (P2):} Unbounded concentration growth within a confined region suggests the possibility of infinite-time singularity formation as a Dirac mass, with the free energy tending to $-\infty$.
    \item \textbf{not (P1) and not (P2):} The solution may exhibit complex dynamics involving simultaneous concentration growth and mass dispersion, possibly leading to intricate asymptotic patterns.
\end{itemize}

\begin{remark}
With the exception of the first case, the precise asymptotic behavior in the other scenarios depends on further details of the model and the initial data. A rigorous analysis of these cases remains an open challenge.
\end{remark}

\subsection{Similarity solutions}\label{sec5}

We now use similarity solutions to describe the dominant dynamics of \eqref{2}. Consider rescaling the length, height, and time scales of the solution according to
\begin{align}
 x=L \hat{x},\quad t=T \hat{t},\quad u(x,t)=H \hat{u}(\hat{x},\hat{t}).
\end{align}
Substituting this change of variables into \eqref{2} yields
\begin{align}
\frac{H}{T}\frac{\partial \hat{u}}{\partial \hat{t}}=-\frac{H^2}{L^4}\frac{\partial}{\partial \hat{x}} \left( \hat{u} \frac{\partial^3 \hat{u}}{\partial \hat{x}} \right)- m \frac{H^{m+1}}{L^2} \frac{\partial}{\partial \hat{x}} \left( \hat{u}^m\frac{\partial \hat{u}}{\partial \hat{x}} \right).
\end{align}
Balancing the coefficients of the second- and fourth-order spatial terms shows that the film thickness scale is inversely proportional to the horizontal length scale: $H=L^{-\frac{2}{m-1}}$. Furthermore, balancing the time derivative term to obtain the distinguished limit gives the relation between the length scale and the time scale:
\begin{align}
L=T^{\frac{m-1}{4m-2}},\quad H=T^{-\frac{1}{2m-1}}.
\end{align}
These scalings produce two classes of similarity solutions:
\begin{alignat*}{2}
&\text{(i) Infinite-time spreading solutions:} \quad && H \to 0 \text{ and } L \to \infty \text{ as } T \to \infty, \\
&\text{(ii) Finite-time blow-up solutions:} \quad && H \to \infty \text{ and } L \to 0 \text{ as } T \to 0.
\end{alignat*}

\section{Conclusions}\label{sec6}

In this paper, we have investigated the dynamical behavior of a one-dimensional thin-film equation in the supercritical regime $3<m<\infty$, where the competition between fourth-order repulsive diffusion and aggregative nonlinearity leads to rich threshold phenomena. A key discovery of the present work is that the extremal function of the sharp Sz.-Nagy inequality, under appropriate mass and norm constraints, serves as the global minimizer of the free energy functional. This variational identity provides a natural and rigorous foundation for determining critical initial data.

By exploiting this variational characterization, we establish a sharp and optimal threshold criterion that classifies solutions according to their initial $L^{m+1}$-norm relative to the critical profile. The result reveals a complete dichotomy: initial data with sufficiently high concentration leads to finite-time blow-up, while data below the threshold yields global existence accompanied by infinite-time spatial spreading. Our threshold significantly extends classical results by allowing positive initial free energy, thereby covering a far broader class of initial configurations.

These findings highlight the crucial role of variational structures and sharp functional inequalities in understanding threshold dynamics for degenerate higher-order diffusion equations. Future directions include the rigorous analysis of long-time behaviors for solutions above the critical energy, the extension of the present variational framework to higher dimensions, and the investigation of threshold phenomena in related gradient-flow models with competing nonlinear effects.

\section*{Data Availability Statement}
All data generated or analysed during this study are included in this published article (and its supplementary information files). All data that support the findings of this study are included within the article (and any supplementary files).

\section*{Statements and Declarations}
This research did not receive any specific grant from funding agencies in the public, commercial, or not-for-profit sectors. The author has no competing interests to declare that are relevant to the content of this article.


\begin{thebibliography}{00}

\bibitem{Tom1} R. Almgren, A. L. Bertozzi, M. P. Brenner, {\it Stable and unstable singularities in the unforced Hele-Shaw cell}, Phys. Fluids, {\bf 8}(6) (1996), 1356-1370.

\bibitem{BP00} A. L. Bertozzi, M. Pugh, {\it Finite-time blow-up of solutions of some long-wave unstable thin film equations}, Indiana Univ. Math. J., {\bf 49}(4) (2000), 1323-1366.

\bibitem{BP98} A. L. Bertozzi, M. Pugh, {\it Long-wave instabilities and saturation in thin film equations}, Comm. Pure Appl. Math., {\bf 51}(6) (1998), 625-661.


\bibitem{BL13} S. Bian, J.-G. Liu, {\it Dynamic and steady states for multi-dimensional Keller-Segel model with diffusion exponent $m > 0$}, Comm. Math. Phys., {\bf 323} (2013), 1017-1070.


\bibitem{BCL09} A. Blanchet, J. A. Carrillo, P. Lauren\c{c}ot, {\it Critical mass for a Patlak-Keller-Segel model with degenerate diffusion in higher dimensions}, Calc. Var. Partial Differential Equations, {\bf 35} (2009), 133-168.

\bibitem{BCM08} A. Blanchet, J. A. Carrillo, N. Masmoudi, {\it Infinite time aggregation for the critical Patlak-Keller-Segel model in $\R^2$}, Comm. Pure Appl. Math., {\bf 61} (2008), 1449-1481.

\bibitem{BCKS99} M. P. Brenner, P. Constantin, L. P. Kadanoff, A. Schenkel, S. C. Venkataramani, {\it Diffusion, attraction and collapse}, Nonlinearity, {\bf 12} (1999), 1071-1098.


\bibitem{B02} C. J. Budd, {\it Asymptotics of multibump blow-up self-similar solutions of the nonlinear Schr\"odinger equation}, SIAM J. Appl. Math., {\bf 62}(3) (2002), 801-830.


\bibitem{C03} T. Cazenave, {\it Semilinear Schr\"odinger equations}, Courant Lecture Notes in Math., vol. 10, New York Univ., Courant Inst. Math. Sci., New York; Amer. Math. Soc., Providence, RI, 2003.

\bibitem{C67} S. Chandrasekhar, {\it An Introduction to the Study of Stellar Structure}, Dover, New York, 1967.

%\bibitem{C84} C. V. Coffman, {\it A nonlinear boundary value problem with many positive solutions}, J. Differential Equations, {\bf 54}(3) (1984), 429-437.

\bibitem{Tom23} P. Constantin, T. F. Dupont, R. E. Goldstein, L. P. Kadanoff, M. J. Shelley, S.-M. Zhou, {\it Droplet breakup in a model of the Hele-Shaw cell}, Phys. Rev. E, {\bf 47}(6) (1993), 4169-4181.


\bibitem{Tom25} T. F. Dupont, R. E. Goldstein, L. P. Kadanoff and S.-M. Zhou, {\it Finite–time singularity
formation in Hele Shaw systems}, Phys. Rev. E, {\bf 47}(6) (1993), 4182-4196.

\bibitem{Tom26} P. Ehrhard, {\it The spreading of hanging drops}, J. Colloid Interface Sci., {\bf 168}(1) (1994), 242-246.

\bibitem{Jose37} J. Evans, V. Galaktionov and J. King, {\it Blow-up similarity solutions of the fourth-order unstable thin film equation},
European J. Appl. Math. {\bf 18}(2) (2007), 195-231.

\bibitem{Jose38} J. Evans, V. Galaktionov and J. King, {\it Source-type solutions of the fourth-order unstable thin film equation},
European J. Appl. Math. {\bf 18}(3) (2007), 273-321.


\bibitem{F01} G. Fibich, {\it Self-focusing in the damped nonlinear Schr\"odinger equation}, SIAM J. Appl. Math., {\bf 61}(5) (2001), 1680-1705.


\bibitem{Tom36} R. E. Goldstein, A. I. Pesci, M. J. Shelley, {\it Topology transitions and singularities in viscous flows}, Phys. Rev. Lett., {\bf 70}(20) (1993), 3043-3046.

\bibitem{JL92} W. J\"{a}ger, S. Luckhaus, {\it On explosions of solutions to a system of partial differential equations modeling chemotaxis,} Trans. Amer. Math. Soc., {\bf 239}(2) (1992), 819-821.


\bibitem{KS70} E. F. Keller, L. A. Segel, {\it Initiation of slime mold aggregation viewed as an instability}, J. Theoret. Biol., {\bf 26} (1970), 399-415.

\bibitem{Tom44} R. S. Laugesen, M. C. Pugh, {\it Energy levels of steady states for thin-film-type equations}, J. Differential Equations, {\bf 182}(2) (2002), 377-415.

\bibitem{Tom42} R. S. Laugesen, M. C. Pugh, {\it Linear stability of steady states for thin film and Cahn-Hilliard type equations}, Arch. Ration. Mech. Anal., {\bf 154}(1) (2000), 3-51.

\bibitem{Tom43} R. S. Laugesen, M. C. Pugh, {\it Properties of steady states for thin film equations}, European J. Appl. Math., {\bf 11}(3) (2000), 293-351.

\bibitem{lieb202} E. H. Lieb, M. Loss, {\it Analysis}, Grad. Stud. Math., vol. 14, Amer. Math. Soc., Providence, RI, 2nd ed., 2001.

    
\bibitem{LW17DC} J.-G. Liu, J. Wang, {\it Global existence for a thin film equation with subcritical mass}, Discrete and Continuous Dynamical systems-Series B, {\bf 22} (2017), 1461-1492.

\bibitem{M89} F. Merle, {\it Limit of the solution of a nonlinear Schr\"odinger equation at blow-up time}, J. Funct. Anal., {\bf 84}(1) (1989), 201-214.


\bibitem{Sz.-Nagy41} B. Sz.-Nagy, {\it \"Uber Integralungleichungen zwischen einer Funktion und ihrer Ableitung}, Acta Univ. Szeged. Sect. Sci. Math., {\bf 10} (1941), 64-74.

\bibitem{P07} B. Perthame, {\it Transport Equations in Biology}, Birkh\"auser, Basel, 2007.


\bibitem{PS98} P. Pucci, J. Serrin, {\it Uniqueness of ground states for quasi-linear elliptic operators}, Indiana Univ. Math. J., {\bf 47} (1998), 501-528.

\bibitem{Jose59} D. Slepčev, {\it Linear stability of self similar solutions of unstable thin-film equations}, Interfaces Free Bound. {\bf 11}(3) 
(2009), 375-398.

\bibitem{Jose60} D. Slepčev, M. C. Pugh, {\it Self similar blow up of unstable thin-film equations}, Indiana Univ. Math. J., {\bf 54}(6) (2005), 1697-1738.

\bibitem{SS99} C. Sulem, P.-L. Sulem, {\it The nonlinear Schr\"odinger equation}, Appl. Math. Sci., vol. 139, Springer, New York, 1999.


\bibitem{W83} M. I. Weinstein, {\it Nonlinear Schr\"odinger equations and sharp interpolation estimates}, Comm. Math. Phys., {\bf 87}(4) (1983), 567-576.

\bibitem{Tom62} M. B. Williams, S. H. Davis, {\it Nonlinear theory of film rupture}, J. Colloid Interface Sci., {\bf 90}(1) (1982), 220-228.

\bibitem{Tom65} T. P. Witelski, A. J. Bernoff, {\it Stability of self-similar solutions for van der Waals
driven thin film rupture}, Phys. Fluids, {\bf 11}(9) (2000), 2443-2445. 

\bibitem{WBB04} T. P. Witelski, A. J. Bernoff, A. L. Bertozzi, {\it Blow up and dissipation in critical-case unstable thin film equation}, European J. Appl. Math., {\bf 15} (2004), 223-256.    
    
\bibitem{Tom67} W. W. Zhang, J. R. Lister, {\it Similarity solutions for van der Waals rupture of a thin
film on a solid substrate}, Phys. Fluids, {\bf 11}(9), (1999), 2454-2462.    
\end{thebibliography}
\end{document}